\documentclass[11pt]{amsart}
\usepackage{graphicx}
\usepackage{amsmath}
\usepackage{xcolor}
\usepackage[active]{srcltx}
\usepackage{fancyhdr}

\usepackage{hyperref}

\makeatletter
\renewcommand*\subjclass[2][2000]{%
  \def\@subjclass{#2}%
  \@ifundefined{subjclassname@#1}{%
    \ClassWarning{\@classname}{Unknown edition (#1) of Mathematics
      Subject Classification; using '1991'.}%
  }{%
    \@xp\let\@xp\subjclassname\csname subjclassname@#1\endcsname
  }%
}
\makeatother
\usepackage{enumerate,url,amssymb,  mathrsfs}

\newtheorem{theorem}{Theorem}[section]
\newtheorem{lemma}[theorem]{Lemma}
\newtheorem*{lemma*}{Lemma}

\newtheorem*{remark*}{Remark}

\theoremstyle{definition}

\theoremstyle{remark}
\newtheorem{remark}[theorem]{Remark}

\numberwithin{equation}{section}

\DeclareMathOperator{\diam}{\diam}

\def\XXint#1#2#3{{\setbox0=\hbox{$#1{#2#3}{\int\limits}$}
\vcenter{\hbox{$#2#3$}}\kern-.5\wd0}}

\begin{document}

\title{H\"older continuity of quasiconformal harmonic mappings from the unit ball to a spatial domain with $C^1$ boundary}


\address{University of Montenegro, Faculty of Natural Sciences and
Mathematics, Cetinjski put b.b. 81000 Podgorica, Montenegro}
\email{antondj@ucg.ac.me}

 \author{Anton Gjokaj}

\begin{abstract}
We prove that every quasiconformal mapping from the harmonic $\beta$-Bloch space between the unit ball and a spatial domain with $C^1$ boundary is globally $\alpha$-H\"older continuous for $\alpha<1-\beta$, with the H\"older coefficient that does not depend neither on the mapping nor on $\beta$. An analogous result also holds for Lipschitz continuous, quasiconformal harmonic mappings for $\alpha <1$. This extends some results from the complex plane obtained by Warschawski in \cite{Warschawski} for conformal mappings and Kalaj in \cite{Kalaj6} for quasiconformal harmonic mappings. 
\end{abstract}
\maketitle

\section{Introduction} 

Let $U$ be a domain in the standard Euclidean space $\mathbf{R}^n$. We say $f=(f_1,\ldots,f_n):U\to \mathbf{R}^n$ is a harmonic mapping if for $j=1,\ldots,n$, $f_j$ is a real-valued harmonic function in $U$, i.e. it satisfies the Laplace equation $\Delta f_j=0.$

We denote by $\mathbf{B}$ the unit ball  $\{x \in \mathbf{R}^n: |x|<1\}$ in $\mathbf{R}^n$ and by $S$ the unit sphere $\{x \in \mathbf{R}^n: |x|=1\}$.
Let $P:\mathbf{B}\times S \to \mathbf{R}$ 
\[ P(x,\xi)=\frac{1-|x|^2}{|x-\xi|^n}\]
be the Poisson kernel for $\mathbf{B}$ and
\[ P[u](x)=\int\limits_{S} P(x,\xi) u(\xi) d\sigma (\xi) \]
be the Poisson integral of a continuous function $u$ on $S$, where $\sigma$ denotes the normalized surface-area measure on $S$. Then $P[u](x)$ is the harmonic function on $\mathbf{B}$, which is continuous on $\overline{\mathbf{B}}$ and which agrees with $u$ on $S$. In case we consider real harmonic functions $u$ on $\mathbf{B}$ with continuous extension to the boundary, it is useful to express them as
\[ u=P[u|_{S}](x).\]

Let $\alpha \in (0,1)$. A function $F:U\subset \mathbf{R}^n \to \mathbf{R}^m$ is said to be $\alpha$-H\"older continuous, $F\in C^{\alpha}(U)$ if
\[ \sup_{x,y\in U, x\neq y} \frac{|F(x)-F(y)|}{|x-y|^{\alpha}}<\infty.\]
The above supremum is called the H\"older coefficient of function $F$ on $U$.

A homeomorphism $f:U\to V$, where $U,V$ are domains in $\mathbf{R}^n$, will be called $K$ quasiconformal (see \cite{Vaisala}) $(K\geq 1)$ if $f$ is absolutely continuous on lines (i.e. absolutely continuous in almost every segment parallel to some of the coordinate axes and there exist partial derivatives which are locally $L^{n}$ integrable in $U$) and
\[ |\nabla f(x)|\leq K l(\nabla f(x)),  \]
for all points $x\in U$, where
\[ |\nabla f(x)|=\sup\{|f'(x)h|: |h|=1\}, \]
\[ l(\nabla f(x))=\inf\{|f'(x) h|: |h|=1\}. \]
 
Let $\beta \in (0,1)$. The harmonic $\beta$-Bloch space is the space of harmonic functions $u=u(x):\mathbf{B}\subset \mathbf{R}^n \to \mathbf{R}$ with the property 
\begin{equation*}  \sup\limits_{x\in \mathbf{B}}  (1-|x|)^{\beta}|\nabla u(x)|<\infty. 
\end{equation*}

We say that a domain $\Omega \subset \mathbf{R}^n$ has $C^1$ boundary if there is a  $C^1$ diffeomorphism $\Phi:\overline{B}\to \overline{\Omega}$. 
\\[0.2cm]

The motivation for the results of this paper comes primarily from the results in \cite{Warschawski} and \cite{Kalaj6} obtained in the complex plane. In \cite{Warschawski}, for every $\alpha \in (0,1)$ it was proved the $\alpha$-H\"older continuity of conformal mappings $f$ from the unit disk onto a Jordan domain $\Omega$ with $C^1$  boundary curve $\gamma$, fixing the origin. The H\"older coefficient in that case depends on properties of $\gamma$ and $\alpha$. In \cite{Kalaj6} this result was extended to the class of quasiconformal harmonic mappings $f$ between Jordan domains $D$ and $\Omega$ with $C^1$ boundaries, with the H\"older coefficient not depending on $f$, but on the properties of boundary curves, on the coefficient of quasiconformality and on two fixed points, $a\in D$ and $b\in \Omega$, such that $f(a)=b.$ 

Since the work of Martio in \cite{Martio}, the problem of H\"older and Lipschitz continuity of harmonic or/and quasiconformal mappings between domains with given properties has aroused a large interest for experts in this theory. For the complex plane there are many interesting results concerning this topic with different settings on domain and codomain.  One of the most important results is in \cite{Pavlovic1} where Pavlovic proved the bi-Lipschitz continuity of harmonic quasiconformal mapping from the unit disk onto itself. Other interesting results can be found in \cite{Kalaj1}, \cite{Kalaj2}, \cite{Kalaj5}, \cite{KalajMat},\cite{Manojlovic}, \cite{MartioNaki}, \cite{MatVuo}, \cite{Partyka}, \cite{Zhu} and references cited therein.

Analogous problems in the space are much more complicated because
of the lack of the techniques of complex analysis. Indeed, the family of conformal mappings on the space coincides with M\"obius transformations. Therefore, for higher dimensional case there are fewer results. Astala and Manojlovic in \cite{Astala} proved the Lipschitz continuity of quasiconformal harmonic mappings of the unit ball onto itself in $\mathbf{R}^n$, partly extending the result from \cite{Pavlovic1}, and the bi-Lipschitz property for $n=3$ in case the mapping is also a harmonic gradient mapping. Kalaj in \cite{Kalaj3} proved the Lipschitz continuity of quasiconformal mappings, satisfying Poisson differential inequality, from the unit ball onto a domain with $C^2$ smooth boundary. In \cite{Kalaj7} it was proven the Lipschitz continuity of harmonic quasiconformal mappings 
$\{ u=P[f]: f\in C^{1,\alpha}, u(0)=0\}$ of the unit ball onto itself.  In \cite{GjoKal} it was proven the Lipschitz continuity of harmonic quasiconformal mappings from the unit ball to a spatial domain with $C^{1,\alpha}$ boundary.  
For other results see also \cite{Arsenovic}, \cite{AKM}, \cite{MatVuo}.

Our goal is to extend the results from \cite{Warschawski} and \cite{Kalaj6} in an appropriate way in space. The result reads as follows. 

\begin{theorem} \label{teorema}
Let $\Omega$ be a spatial domain with $C^1$ boundary, $b\in \Omega$ and $\beta \in (0,1)$ fixed. For every $\alpha \in (0,1-\beta)$ and $K \geq 1$, there is a constant $M=M(\alpha, b, K, \Omega)$ such that every K-quasiconformal mapping $f:\mathbf{B}\to \Omega$ that belongs to the harmonic $\beta$-Bloch space and satisfies $f(0)=b$ is $\alpha$-H\"older continuous with H\"older coefficient $M$, i.e.:
\[ |f(x)-f(y)|\leq M |x-y|^{\alpha}, \] 
for all $x,y \in \mathbf{B}$. 
\end{theorem} 

We give an additional result. The condition of Bloch space in Theorem \ref{teorema} can be replaced with Lipschitz continuity, attaining $\alpha$-H\"older continuity, in this case for $\alpha \in (0,1)$, but again with "uniform" H\"older coefficient. 

\begin{theorem} \label{teorema2} Let $\Omega$ be a spatial domain with $C^1$ boundary and $b\in \Omega$. For every $\alpha \in (0,1)$ and $K \geq 1$, there is a constant $M=M(\alpha, b, K, \Omega)$ such that every K-quasiconformal harmonic mapping $f:\mathbf{B}\to \Omega$ which is Lipschitz continuous in $\mathbf{B}$ and satisfies $f(0)=b$ is $\alpha$-H\"older continuous with H\"older coefficient $M$.

\end{theorem}

\section{Auxiliary results} 

We collect some important results. The next lemma is an extension of the well-known Hardy-Littlewood theorem (see \cite[page 411-414]{Goluzin}).
\begin{lemma}\label{lema1}
Let $u:\overline{\mathbf{B}}\subset \mathbf{R}^n \to \mathbf{R}$, be a real harmonic function in $\mathbf{B}$ with continuous extension in $\overline{\mathbf{B}}$, and $\alpha \in (0,1)$.  
We define
$$X=\sup_{ \xi,\eta \in S; \xi\neq\eta} \frac{|u(\xi)-u(\eta)|}{|\xi-\eta|^{\alpha}}$$ and 
$$Y=\sup_{x \in \mathbf{B}} (1-|x|)^{1-\alpha} |\nabla u(x)|.$$
Then there is a constant $C=C(\alpha)>0$ such that 
$$\frac1C X \leq Y \leq C X.$$
\end{lemma}

For the proof of the second inequality see \cite[Theorem 2.1]{GjoKal}  - part of it will also be shown in the proof of Theorem \ref{teorema}; for the first see \cite[Theorem 2]{Pavlovic} and \cite[Lemma 2.5]{GjoKal}.
\\

For the proof of our theorems, it will be necessary to "control" the modulus of continuity of $f$ on $S$, uniformly (not depending on a specific $f$). Using the $C^1$ condition on the boundary of $\Omega$ and the well-known Mori's theorem (see \cite{Mori}) we get the following lemma.

\begin{lemma} \label{lema2}
Let $f$ be a $K$-quasiconformal mapping from the unit ball $\mathbf{B}$ in $\mathbf{R}^n$ onto a spatial domain $\Omega\subset\mathbf{R}^n$ with $C^1$ boundary. Then there is a constant $M=M(K,f(0),\Omega)$ and $\alpha=\alpha(K,f(0),\Omega)<1$ such that 
\[ |f(x)-f(y)| \leq M |x-y|^{\alpha}, \]
for all $x,y\in \mathbf{B}$.
\end{lemma}

\begin{proof}
Since $\Omega$ has $C^1$ boundary, there is a diffeomorphism (up to the boundary) $\Phi:\overline{\mathbf{B}} \to \overline{\Omega}$, such that $\Phi(0)=f(0)$. The mappings $\Phi$ and $\Phi^{-1}$ are quasiconformal, since $\Phi$ is bi-Lipschitz. \\Notice that the mapping $F =\Phi^{-1}\circ f$ is a $K_1=K_1(K,\Omega,f(0))$ quasiconformal mapping of the unit ball onto itself, such that $F(0)=0$. By Mori's theorem, we have 
\[ |F(x)-F(y)| \leq M_1 |x-y|^{\alpha} \]
where $\alpha=K_1^{1/(1-n)}$ and $M_1=M_1(K_1,\Omega, f(0))$. 
Since $\Phi$ is Lipschitz continuous, it follows that $f=\Phi \circ F$ is a $\alpha$-H\"older continuous mapping with H\"older coefficient $M$ which depends on $K, f(0)$ and $\Omega$. 
\end{proof}

\section{ Proofs of Theorem \ref{teorema} and Theorem \ref{teorema2}}
\begin{proof}[Proof of Theorem \ref{teorema}]
Let $q=f(\eta_0) \in \partial \Omega$ ($\eta_0\in S$ will be determined later). Since $\Omega$ has $C^1$ boundary, postcomposing $f$ with an isometry  (isometries preserve the given properties of $f$), we can assume that $q=(0,\ldots,0)$ and that the normal vector of tangent plane at $\partial \Omega$ is $n_q=(0,\ldots,0,-1)$. This allows us to express a neighbourhood of $q=0$ in $\partial \Omega$, as a graph of a $C^{1}$ function $\Psi: B_{n-1}(0,\delta_1) \subset \mathbf{R}^{n-1}\to \mathbf{R}$, i.e. points of $\partial \Omega$ in the neighbourhood of $q=0$ are of the form
\begin{equation} \label{1} (\zeta_1,\zeta_2, \ldots, \zeta_{n-1}, \Psi(\zeta_1,\ldots, \zeta_{n-1})),
\hspace{0.3cm} \textit{$(\zeta_1,\ldots,\zeta_{n-1})\in B_{n-1}(0,\delta_1)$}
,\end{equation}
where $\Psi(0,\ldots,0)=0$ and $D_j\Psi(0,\ldots,0)=0$, for $j \in \{1,\ldots,n-1\}$.  
For every $\epsilon>0$ there is a $\delta_2$ $(0<\delta_2\leq\delta_1)$ such that 
\begin{equation} \label{2}  |\Psi(\zeta_1, \ldots, \zeta_{n-1})-\Psi(0,\ldots, 0)-\sum\limits_{j=1}^{n-1} D_j\Psi(0,\ldots,0)\zeta_j| \leq \epsilon |(\zeta_1,\ldots, \zeta_{n-1})|,\end{equation} 
for $(\zeta_1,\ldots,\zeta_{n-1}) \in B_{n-1}(0,\delta_2)$. Constant $\delta_2$ depends on $\epsilon$ and $\partial \Omega$, but since $\partial \Omega$ is a compact and $C^1$ surface, it does not depend on a specific point ($q$) of $\partial \Omega$. \\
As $f$ is K-quasiconformal, from Lemma $\ref{lema2}$ we can choose $\delta$ such that \begin{equation} \label{3} |f(\xi)-f(\eta_0)|\leq \delta_2, \end{equation} for $\xi \in S_\delta(\eta_0)=\{ \xi \in S: |\xi-\eta_0|<\delta\}$, where the constant $\delta$ depends on $\delta_2, K, f(0)$ and $\Omega$. \\ \\
From $(\ref{3})$ we have 
\[ |(f_1(\xi),\ldots, f_{n-1}(\xi))-\underbrace{(f_1(\eta_0),\ldots, f_{n-1}(\eta_0))}_\textit{$(0,\ldots,0)$}|<\delta_2, \]
for $\xi \in S_\delta(\eta_0)$, which with $(\ref{1})$ and $(\ref{2})$ implies that $f_n(\xi)=\Psi(f_1(\xi),\ldots, f_{n-1}(\xi))$ and 
\begin{equation} \label{4}
|f_n(\xi)|\leq \epsilon |(f_1(\xi),\ldots, f_{n-1}(\xi))|
\end{equation}
for $\xi\in S_\delta(\eta_0)$. 
\\
 Using the Poisson integral formula we have
\[ f_n(x)=\int\limits_{S} \frac{1-|x|^2}{|x-\xi|^n} f_n(\xi) d\sigma (\xi).\]
Observe that
\begin{equation*} \label{001}
\nabla f_n(x) =\int\limits_{S} Q(x,\xi) f_n(\xi) d\sigma (\xi),
\end{equation*}
where\footnote{For detailed steps see the proof of \cite[Theorem 2.1]{GjoKal}.}
\begin{equation} \label{002} Q(x,\xi)=\left(-2x+\frac{-n(1-|x|^2)(x-\xi) }{1+|x|^2-2\langle \xi, x \rangle}\right) \frac{1}{(1+|x|^2-2\langle \xi, x \rangle)^{\frac{n}{2}}}.\end{equation}

Let $h\in R^n$ be an arbitrary vector.
Then
\begin{equation} \label{0011}
\begin{split}
\langle \nabla f_n(x), h \rangle &=\int\limits_{S} \langle Q(x,\xi), h \rangle f_n(\xi) d\sigma (\xi).
\\
&=\int\limits_{S} \langle Q(x,\xi), h\rangle [f_n(\xi)-f_n(\eta_0)] d\sigma (\xi).
\end{split}
\end{equation}

On the other hand
\begin{equation}\label{004}
\begin{split}
& \left|-2\langle x, h\rangle+\frac{-n(1-|x|^2)\langle x-\xi, h \rangle }{1+|x|^2-2\langle \xi, x \rangle}\right| \\
& \leq 2|x||h|+n\frac{(1-|x|^2)|x-\xi||h|}{|x-\xi|^2}\leq  \\
& = 2|x||h| +2n|h| \underbrace{\frac{1-|x|}{|x-\xi|}}_\textit{$\leq 1$} \leq (2+2n)|h|.
\end{split}
\end{equation}

From  $(\ref{002}), (\ref{0011}), (\ref{004})$ we get
\begin{equation*}
|\langle \nabla f_n(x), h \rangle|\leq (2n+2)|h| \int\limits_{S} \frac{|f_n(\xi)-f_n(\eta_0)|}{|x-\xi|^n} d\sigma (\xi).
\end{equation*}
Since $h$ was taken arbitrary, then
\begin{equation} \label{15}
|\nabla f_n(x)|\leq (2n+2) \int\limits_{S} \frac{|f_n(\xi)-f_n(\eta_0)|}{|x-\xi|^n} d\sigma (\xi),
\end{equation}
for all $x\in \mathbf{B}$.
\\ Let $\alpha \in (0,1-\beta)$. We consider now 
\begin{equation} \label{5} A=\sup\limits_{x\in \mathbf{B}} (1-|x|)^{1-\alpha}|\nabla f_n(x)|. \end{equation} 
Since $f$ (and so $f_n$) belongs to the family of harmonic $\beta$- Bloch space functions, we have
\begin{equation} \sup\limits_{x\in \mathbf{B}} (1-|x|)^{\beta}|\nabla f_n(x)|<\infty. \end{equation} 
Since $(1-|x|)^{1-\alpha}|\nabla f_n(x)|=(1-|x|)^\beta|\nabla f_n(x)| (1-|x|)^{1-\alpha-\beta}$, where $1-\alpha-\beta>0$, we have that the function $(1-|x|)^{1-\alpha} |\nabla f_n(x)|$ can be extended continuously in $\overline{\mathbf{B}}$, vanishing on $S$. This means that $A$ is attained in an interior point of $\mathbf{B}$. Let this point be $r\eta_0$, where $r\in [0,1)$ and $\eta_0 \in S$ is the point considered from the begin of the proof. 
Considering $$A_j=\sup\limits_{x\in \mathbf{B}} (1-|x|)^{1-\alpha}|\nabla f_j(x)|,$$ the same conclusion can be drawn for functions $f_j$, $j\in \{1,\ldots, n-1\}$. Although, we want to get an upper bound of $A_j$, $j\in \{1,\ldots,{n-1}\}$, in term of $A$. \\
Since $f$ is a K-quasiconformal mapping we have that 
\begin{equation} \label{25} \max\limits_{|h_1|=1} |f'(x)h_1| \leq K \min\limits_{|h_2|=1} |f'(x)h_2| ,\end{equation} for all $x\in \mathbf{B}$. 
In case $f'(x)$ is a singular matrix, then (\ref{25}) implies that $f'(x)$ is a zero matrix.
On the other side, it is a well known fact that for nonsingular matrix $P$ we have
$\min\limits_{|h|=1}|Ph|^2=\min\limits_{|k|=1} |P^Tk|^2$ and $\max\limits_{|h|=1}|Ph|^2=\max\limits_{|k|=1} |P^Tk|^2$. Indeed,
\begin{equation*} \begin{split} \max\limits_{|h|=1} |Ph|^2 &=\max_{|h|=1} \langle Ph,Ph\rangle = \max\limits_{|h|=1} \langle P^TPh,h \rangle 
\\
&= \max\{ \lambda: \exists h \neq 0, P^TPh=\lambda h\} \\
& = \max\{ \lambda: \exists h \neq 0, PP^TPh=\lambda Ph \} \\
& =\max \{ \lambda: \exists k \neq 0, PP^Tk=\lambda k\} \\
&=\max\limits_{|k|=1} \langle PP^Tk,k \rangle =\max\limits_ {|k|=1} |P^Tk|^2. 
\end{split}
\end{equation*}
Analogously for "$\min$" case. In both cases (either $f'(x)$ is a singular or nonsingular matrix), combining $(\ref{25})$ and the last two identities, it follows that 
\begin{equation} \label{31} \max\limits_{|k_1|=1} |f'(x)^Tk_1|\leq K \min\limits_{|k_2|=1} |f'(x)^Tk_2|.\end{equation} 
\\Taking $k_1=e_j$ and $k_2=e_n$ in $(\ref{31})$ we get
\[ |\nabla f_j(x)| \leq K |\nabla f_n(x)|, \hspace{0.5cm} \textit{for $x\in \mathbf{B}$}, \]
which implies
$$A_j \leq K A,$$
for $j\in \{1,\ldots, n-1\}$. 
In view of Lemma $\ref{lema1}$ we have
\begin{equation} \label{6} |f_j(\xi)-f_j(\eta)|\leq AKC(\alpha)|\xi-\eta|^{\alpha}, \end{equation}
for $j\in \{1,\ldots, n-1\}$ and 
$$|f_n(\xi)-f_n(\eta)|\leq AC(\alpha)|\xi-\eta|^{\alpha},$$
for all $\xi,\eta \in S$. \\
The last two inequalities imply 
\begin{equation} \label{12}
|f(\xi)-f(\eta)|\leq AC(\alpha)\sqrt{1+K^2(n-1)} |\xi-\eta|^{\alpha},
\end{equation}
for $\xi,\eta \in S$.  \\
Our next and final goal is to prove that $A$ does not depend on $f$, but on the given parameters from the statement of the theorem. 
From $(\ref{15})$ we have
\begin{equation}\label{8}  \begin{split} 
A&=(1-r)^{1-\alpha}|\nabla f_n(r\eta_0)|\\&  \leq (2n+2) (1-r)^{1-\alpha}\int_S \frac{|f_n(\xi)-f_n(\eta_0)|}{|r\eta_0-\xi|^n}d\sigma(\xi) 
\\ & \leq (2n+2)\underbrace{\int\limits_{S_\delta(\eta_0)} \frac{(1-r)^{1-\alpha}|f_n(\xi)-f_n(\eta_0)|}{|r\eta_0-\xi|^n}d\sigma(\xi)} _\textit{$I_1$} \\&+
(2n+2) \underbrace{\int\limits_{S\backslash S_\delta(\eta_0)} \frac{(1-r)^{1-\alpha}|f_n(\xi)-f_n(\eta_0)|}{|r\eta_0-\xi|^n}d\sigma(\xi)}_\textit{$I_2$}
\end{split}
\end{equation}
For the first integral, using $(\ref{4})$, $(\ref{6})$ we get the following estimations
\begin{equation}\label{7} \begin{split}
I_1&=(1-r)^{1-\alpha}\int\limits_{S_\delta(\eta_0)} \frac{|f_n(\xi)-f_n(\eta_0)|}{|r\eta_0-\xi|^n}d\sigma(\xi)  
\\ & \leq \epsilon (1-r)^{1-\alpha}\int\limits_{S_\delta(\eta_0)} \frac{|(f_1(\xi),\ldots,f_{n-1}(\xi))|}{|r\eta_0-\xi|^n}d\sigma(\xi) 
\\& \leq \epsilon \sqrt{n-1}AKC(\alpha)(1-r)^{1-\alpha}\int\limits_{S_\delta(\eta_0)} \frac{|\xi-\eta_0|^{\alpha}}{|r\eta_0-\xi|^n}d\sigma(\xi)
\\& \leq \epsilon \sqrt{n-1}AKC(\alpha)(1-r)^{1-\alpha}\int\limits_{S} \frac{|\xi-\eta_0|^{\alpha}}{|r\eta_0-\xi|^n}d\sigma(\xi).
\end{split}
\end{equation} 
\textbf{1st case} $r=|r\eta_0|\geq \frac{1}{2}$ \\
It is obvious that the last integral does not depend on the point $\eta_0 \in S$, so we can assume that $\eta_0=(1,0,\ldots,0)$. Moreover, the last integrand function in $(\ref{7})$ depends only on the first coordinate of $\xi$, so we use the following representation \cite[Appendix A5]{Axler}:
\begin{equation} \label{19} 
\int\limits_{S} \frac{|\xi-\eta_0|^{\alpha}}{|r\eta_0-\xi|^n}d\sigma(\xi)= C_1 \int_{-1}^{1} \int\limits\limits_{S_{n-2}} \frac{(2-2\xi_1)^{\frac{\alpha}{2}}(1-\xi_1^2)^{\frac{n-3}{2}}}{((1-r)^2+r(2-2\xi_1))^{\frac{n}{2}}}  d\sigma_{n-2}(\zeta)d\xi_1,  \end{equation}
where $\sigma_{n-2}$ denotes the respective normalized surface-area measure on the unit sphere $S_{n-2}$ in $\mathbf{R}^{n-1}$. The constant $C_1$ depends on $n$ and the volumes of the unit balls in $\mathbf{R}^n$ and $\mathbf{R}^{n-1}$. 
\\The constant $\widetilde{C}$, in what follows, can change its value and depends on $\alpha$. From ($\ref{19}$), using Fubini's theorem, the assumption $r\geq \frac12$ and some elementary inequalities, it follows
\begin{equation*}
\begin{split} 
\int\limits_{S} \frac{|\xi-\eta_0|^{\alpha}}{|r\eta_0-\xi|^n}d\sigma(\xi) 
&\leq C_1 \int_{-1}^{1}  \frac{(2-2\xi_1)^{\frac{\alpha}{2}}}{(1-r)^2+r(2-2\xi_1)} \frac{2^{\frac{n-3}{2}}(1-\xi_1)^{\frac{n-3}{2}}}{{((1-r)^2+r(2-2\xi_1))^{\frac{n-2}{2}}}  } d\xi_1 \\
& = C_1 2^{\frac{\alpha}{2}} 2^{\frac{n-3}{2}} \int_{-1}^{1}  \frac{(1-\xi_1)^{\frac{\alpha-1}{2}}}{(1-r)^2+r(2-2\xi_1)} \left(\frac{1-\xi_1}{(1-r)^2+r(2-2\xi_1) }\right)^{\frac{n-2}{2}} d\xi_1
\\ & \leq
\widetilde{C} \int_{-1}^{1}  \frac{(1-\xi_1)^{\frac{\alpha-1}{2}}}{(1-r)^2+(1-\xi_1)}d\xi_1.  
\end{split} 
\end{equation*}

Using the substitution $t=\frac{\sqrt{1-\xi_1}}{1-r}$, we get
\begin{equation*} \begin{split}
\int\limits_{S} \frac{|\xi-\eta_0|^{\alpha}}{|r\eta_0-\xi|^n}d\sigma(\xi)& \leq
 (1-r)^{\alpha-1} \widetilde{C} \int_{0}^{\frac{\sqrt{2}}{1-r}}  \frac{2t^{\alpha}}{1+t^2}dt\\
&\leq (1-r)^{\alpha-1}\widetilde{C}\int_{0}^{\infty} \frac{t^{\alpha}}{1+t^2}dt.
\end{split}
\end{equation*}

Since the last integral converges we finally have
\begin{equation} \label{006} (1-r)^{1-\alpha} \int\limits_{S} \frac{|\xi-\eta_0|^{\alpha}}{|r\eta_0-\xi|^n}d\sigma(\xi) \leq \widetilde{C}, \end{equation} for $r\in [\frac{1}{2},1)$, where $\widetilde{C}$ depends on $\alpha$ only. \\[0.5cm]

\textbf{2nd case} $r< \frac{1}{2}$
\\ \\
Here the proof is quite straightforward. Since
\begin{equation*} \label{018} \frac{(1-r)^{1-\alpha}|\xi-\eta_0|^{\alpha}}{|r\eta_0-\xi|^n}=\frac{(1-r)^{1-\alpha}|\xi-\eta_0|^{\alpha}}{{((1-r)^2+r|\xi-\eta_0|^2)^{\frac{n}{2}}}}< \frac{2^{\alpha}\cdot 1}{(\frac{1}{2})^n}=2^{n+\alpha}, \end{equation*}
it follows that 
\begin{equation} \label{007}  \int\limits_{S} \frac{(1-r)^{1-\alpha}|\xi-\eta_0|^{\alpha}}{|r\eta_0-\xi|^n}d\sigma(\xi) \leq 2^{n+\alpha}.\end{equation}
\\[0.1cm]
From $(\ref{006})$ and $(\ref{007})$, we conclude that
\begin{equation} \label{11}
\int\limits_{S} \frac{(1-r)^{1-\alpha}|\xi-\eta_0|^{\alpha}}{|r\eta_0-\xi|^n}d\sigma(\xi) \leq \widetilde{C}(\alpha),
\end{equation}
for $r\in [0,1)$, which combined with $(\ref{7})$ gives 
\begin{equation} \label{35}
I_1\leq \epsilon \sqrt{n-1}AKC(\alpha)\widetilde{C}(\alpha).
\end{equation}
\\[0.5cm]
Now, let we focus on 
\[ I_2=(1-r)^{1-\alpha}\int\limits_{S\backslash S_\delta(\eta_0)} \frac{|f_n(\xi)-f_n(\eta_0)|}{|r\eta_0-\xi|^n}d\sigma(\xi) \]
from $(\ref{8})$. It follows that
\[ I_2 \leq diam(\Omega) \frac{\mu(S\backslash S_\delta(\eta_0))}{\mu(S)} \max_{\xi \in S\backslash S_\delta(\eta_0)} \frac{1}{|r\eta_0-\xi|^n} \leq  diam(\Omega) \max_{\xi \in S\backslash S_\delta(\eta_0)} \frac{1}{|r\eta_0-\xi|^n}\]\\
We want to give a lower bound for $|r\eta_0-\xi|^2$, for $\xi \in S\backslash S_\delta(\eta_0)$  in terms of $\delta$. \\
Again, without loss of generality, we can assume that $\eta_0=(1,0,\ldots,0)$. 
Since $\xi=(\xi_1,\ldots,\xi_n) \in S \backslash S_{\delta}(\eta_0)$, we have 
\[ \delta^2 \leq |\xi-\eta_0|^2=2-2\xi_1.\] 
This implies that 
\begin{equation*} \begin{split} |r\eta_0-\xi|^2&=1+r^2-2r\xi_1\geq 1+r^2+r(\delta^2-2) \\ &\geq 1+\left(\frac{2-\delta^2}{2}\right)^2+\frac{2-\delta^2}{2}(\delta^2-2) \\&=\delta^2\left(1-\frac{\delta^2}{4}\right). \end{split} \end{equation*}
It follows that 
\begin{equation} \label{9} I_2 \leq diam(\Omega)\left(\delta\sqrt{1-\frac{\delta^2}{4}}\right)^{-n}.
\end{equation}
From $(\ref{8})$, $(\ref{35})$ and $(\ref{9})$ we get 
\begin{equation} \label{21}
A \leq \epsilon(2n+2)\sqrt{n-1}AKC(\alpha)\widetilde{C}(\alpha)+(2n+2)diam(\Omega)\left(\delta\sqrt{1-\frac{\delta^2}{4}}\right)^{-n}.
\end{equation}
We can choose $\epsilon$ such that 
\[ \epsilon(2n+2)\sqrt{n-1}KC(\alpha)\widetilde{C}(\alpha) \leq \frac{1}{2},\]
so from $(\ref{21})$, we have
\begin{equation} \label{16} A \leq 2(2n+2)diam(\Omega)\left(\delta\sqrt{1-\frac{\delta^2}{4}}\right)^{-n}.\end{equation}

As mentioned before, $\delta$ depends on $K,f(0),\Omega$  and $\alpha$ (since $\epsilon$ depends on $\alpha)$.
Finally, from $(\ref{12})$ and $(\ref{16})$ we conclude that $f$ is $\alpha$-H\"older continuous on $S$ with H\"older coefficient $M$ that does not depend on $f$ or $\beta$. This, by a classical result of Dyakonov (see \cite[Lemma 4]{Dyakonov}), implies our result.

\end{proof}

\begin{proof}[Proof of Theorem \ref{teorema2}]
The proof of this theorem is analogous to the previous one. The only difference has to do with the attainment of A at $(\ref{5})$ in an interior point of $\mathbf{B}$. In this case, it can be proven in the following way. Since $f_n$ is Lipschitz continuous and smooth it follows that $|\nabla f_n(x)| < \infty$, so for every $\alpha\in (0,1)$ the function
\[ (1-|x|)^{1-\alpha}|\nabla f_n(x)| \]
can be extended continuosly in $\overline{\mathbf{B}}$, vanishing on $S$. As in the previous proof, the H\"older coefficient will depend only on $\alpha, f(0), K$ and  $\Omega$.
\end{proof}

\begin{remark} The modulus of continuity in Theorem \ref{teorema2} can be taken to be  a regular majorant $\omega=\omega(t)$ instead of $t^{\alpha}$, since they are "comparable" for some $\alpha \in (0,1)$ (see \cite[Preposition 1] {Pavlovic}). In that case, constant $M$ depends on $\omega$ instead of $\alpha$. 
\end{remark}

\end{document}